\documentclass[a4paper]{article}
\usepackage{amsthm,amssymb,amsmath,enumerate,mathtools} 

\usepackage{hyperref}

\newcommand{\N}{\ensuremath{\mathbb{N}}}

\newcommand{\Z}{\ensuremath{\mathbb{Z}}}

\newcommand{\bbN}{\ensuremath{\mathbb{N}}}

\newcommand{\bbZ}{\ensuremath{\mathbb{Z}}}

\newcommand{\Rcal}{\ensuremath{\mathcal{R}}}
\newcommand{\Ccal}{\ensuremath{\mathcal{C}}}

\newcommand{\Dinf}{\ensuremath{D_\infty}}

\newcommand{\free}{\ast\!\! }

\DeclareMathOperator{\defi}{:\!=}
\DeclareMathOperator{\Sqcup}{\bigsqcup}

\theoremstyle{definition}

\theoremstyle{plain}
\newtheorem{theorem}{Theorem}[section]
\newtheorem{lemma}[theorem]{Lemma}

\newtheorem{thm}[theorem]{Theorem}

\newtheorem{coro}[theorem]{Corollary}
\newtheorem{remark}[theorem]{Remark}

\newtheorem{conj}{Conjecture}
\newtheorem*{conj*}{Conjecture}

\theoremstyle{remark}

\newenvironment{txteq*}
{
	\begin{equation*}
		\begin{minipage}[t]{0.85\textwidth} 
			\em                                
		}
		{\end{minipage}\end{equation*}\ignorespacesafterend}

\title{From cycles to circles in Cayley graphs
}
\author{Babak Miraftab \and Tim R\"uhmann \smallskip \and Department of Mathematics\\
 University of Hamburg\\
 Bundesstra\ss e 55\\
 20146 Hamburg\\
 Germany}
\date{\today}

\begin{document}
 \maketitle

\begin{abstract}
For  locally finite infinite graphs the notion of  Hamilton cycles can be extended to Hamilton circles, homeomorphic images of~$S^1$ in the Freudenthal compactification.
In this paper we extend some well-known theorems of the Hamiltonicity of finite Cayley graphs to infinite Cayley graphs.

 \end{abstract}

\section{Introduction}

In 1959 Elvira Rapaport Strasser~\cite{rap} proposed the problem of studying Hamilton cycles in Cayley graphs for the first time.
In fact the motivation of finding Hamilton cycles in Cayley graphs comes from ``bell ringing" and the ``chess problem of the knight", see \cite{campan1,campan}.
Later  Lov\'{a}sz \cite{BabaiLovasz}  extended this problem from Cayley graphs to vertex-transitive graphs.
He conjectured that every  finite connected vertex-transitive graph  contains  a  Hamilton  cycle except  only five  known  counterexamples, see \cite{BabaiLovasz}.
A famous family of vertex-transitive graphs are Cayley graphs and the conjecture of  Lov\'{a}sz can be formulated for Cayley graphs. 
This is well-known as the weak version of Lov\'{a}sz conjecture:
\begin{conj*}[{Weak  Lov\'{a}sz Conjecture}]
\label{lovasz}
Any Cayley graph of a finite group with three or more elements has a Hamilton cycle.
\end{conj*}	
There is an immense number of papers about this conjecture, but it is still open.
For a survey see \cite{wittesurvey}.
One can look at this conjecture for  infinite transitive graphs and infinite Cayley graphs.
But we first need an appropriate generalization of Hamilton cycles to the setting of  infinite graphs.
We follow the topological approach defined in \cite{TopSurveyI}, also see Section~\ref{topology}; for another instance of this approach see \cite{AgelosFleisch,hamlehpott,Karl1,Karl2}.
The first person to study infinite Hamilton cycles in Cayley graphs was Elvira Rapaport Strasser \cite{rap}, who was searching for two-way Hamilton infinite paths in one-ended groups, which are known as Hamilton circles today.
We state a useful theorem of \cite{rap} which enables us to find a Hamilton circle in Cayley graphs of one-ended group:
\begin{thm}{\rm\cite{rap}}\label{Rapaport}
Let~$G$ be a connected locally finite infinite graph with only one end.
Suppose that there is a set~$R$ of pairwise disjoint cycles containing every vertex precisely once, and a set~$S$ of (not necessary disjoint)~$4$-cycles containing every vertex.
Assume that when two of the~$4$-cycles in~$S$ intersect, the intersection is a single vertex; and when one of the cycles in~$R$ intersects a~$4$-cycle in~$S$, the intersection is a common edge. 
Also assume every edge of~$G$ lies on either one of the cycles in~$R$ or on one of the~$4$-cycles in~$S$(or both).
Then there is a two-way infinite Hamilton path in~$G$.
\end{thm}	
Unfortunately, a direct extension of the weak Lov\'{a}sz conjecture for infinite groups fails.
However Geogakopoulos \cite{AgelosFleisch} conjectured that the weak Lov\'{a}sz conjecture holds true for infinite groups except for groups can be written as a free product with amalgamation of more than $k$ groups over finite subgroup of order~$k$.
In \cite{Cayley_HC_M_T} the authors found a counterexample for the above version of the weak Lov\'{a}sz conjecture for infinite groups.
But it seems that the weak Lov\'{a}sz conjecture for infinite groups holds for infinite groups with at most two-ends  except when the Cayley graph is the double ray \cite{Matthias}.

\begin{conj*}\label{M-R}
Any Cayley graph of a group with at most two ends is Hamiltonian except the double ray.
\end{conj*}

As we mentioned before there are several theorems regarding the existence of Hamilton cycles in finite  Cayley graphs.
In this paper wee study possible generalizations of the following theorems from finite groups to infinite groups:

\begin{itemize}
	\item Every connected Cayley graph on any finite~$p$-group is Hamiltonian.
	\item Let~$G$ be a finite group, generated by three involutions~$a,b,c$ in such a way that~${ab=ba}$.
	Then the Cayley graph~$\Gamma(G, \{a,b,c\})$ is Hamiltonian.
	\item Every finite group~$G$ of size at least~$3$ has a generating  set~$S$ of size at most~${ \log_2 |G|}$, such that~$\Gamma(G,S)$ contains a Hamilton cycle. 
	\item Let~$G=\langle S\rangle$ be finite and let~$N$ be a cyclic normal subgroup of~$G$. 
	If~$[\bar{x_1},\ldots,\bar{x_n}]$  is a Hamilton cycle of~$\Gamma(G/N,\overline{S})$ and the product~$x_1\cdots x_n$ generates~$N$, then~$\Gamma(G,S)$   contains a Hamilton cycle.
\end{itemize}   

 The problem of finding Hamilton circles in infinite Cayley graphs is also a question about the generating set used in creation of the Cayley graph in question, see \cite{Cayley_HC_M_T} for more detail.


This paper is structured in the following manner.
In Section~\ref{prelim} we recall most of the important notations used in this paper. 
Section~\ref{construction} contains results of finding conditions on groups and generating sets such that any Cayley graph of those groups fulfilling those conditions contains a Hamilton circle.
The most noteworthy results in Section~\ref{construction} is Theorem~\ref{index2 hamilton} which is about the connection of Hamilton circles of the Cayley of a group $G$ and Hamilton circles of the Cayley of a subgroup of $G$ with index 2.
Also Theorem~\ref{Z_p2} and Theorem~\ref{split over 2group} are significant in this section.
We study Hamilton circles in any Cayley graph of groups splitting over~$\mathbb Z_2$ in Theorem~\ref{Z_p2}.
 Theorem~\ref{split over 2group} is devoted to Hamilton circles in the Cayley graph of groups splitting groups over~$\Z_p$, where~$p$ is a prime number.
In Section~\ref{Section_rap} we prove that a family of 3-connected planar Cayley graphs admits a Hamilton circle, see Theorem~\ref{inf_rap_k2}.
Section~\ref{Section_add_gen} is used to study Hamiltonicity of Cayley graphs of infinite groups for which one can choose additional generators for those groups.
The main results of Section~\ref{Section_add_gen} are Theorem~\ref{add_a} and Theorem~\ref{inf_group_factor}. 
Theorem~\ref{add_a} says for a given generating set~$S$ of two-ended group~$G$, there is an element~$g\in G$ such that the Cayley graph $G$ with respect of~$S\cup\{g^{\pm}\}$ admits a Hamilton circle.
In addition Theorem~\ref{inf_group_factor} is a generalization of the Factor Group Lemma of finite groups to infinite groups. 


\section{Preliminaries}
\label{prelim}

We follow the notations and the terms of \cite{diestelBook10noEE} for graph-theoretical terms and for group-theoretical \cite{scott} unless stated otherwise.
Section~\ref{topology} is used to give a brief history on the concept of ends and define the topology on graphs we use to define ``infinite cycles'' in graphs.
In Section~\ref{on graphs} we recall the most important definitions related to graphs and the space~$|\Gamma|$ as well as defining an alternative notation for paths and cycles which is useful for Cayley graphs. 
The commonly used definitions related to groups are recapped in Section~\ref{on groups}.
Please note that throughout this paper~$\Gamma$ will be reserved for graphs and~$G$ will be reserved for groups.

\subsection{Topology}
\label{topology}

In 1931, Freudenthal~\cite{freu31} defined the concept of topological ends for topological spaces and topological groups for the first time. 
He defined the following topology on a locally compact Hausdorff space~$X$.
Consider an infinite sequence~$U_1\supseteq U_2 \supseteq \cdots$ of non-empty connected open subsets of~$X$ such that the boundary of each~$U_i$ is compact and~$\bigcap\overline{U_i}=\emptyset$.
Two such sequences~$U_1\supseteq U_2 \supseteq \cdots$ and~$V_1\supseteq V_2 \supseteq \cdots$ are \emph{equivalent} if for every~${i\in\N}$, there are~$j,k\in\N$ in such a way that~$U_i\supseteq V_j$ and~$V_i\supseteq U_k$.
The equivalence classes\footnote{We denote the equivalence class of~$U_i$ by~$[U_i]$.} of those sequences are \emph{topological ends} of~$X$.
The \emph{Freudenthal compactification} of the space~$X$ is the set of ends of~$X$ together with~$X$.
He defined neighborhoods of ends in the following manner:
A neighborhood of an end~${[U_i]}$ is an open set~$V$ such that~$V \supsetneq U_n$ for some~$n$.
We denote the  Freudenthal compactification of the topological space~$X$ by~$|X|$.

Halin \cite{halin64} introduced a different notion of ends, the ends, for infinite graphs in~1964.
A \emph{ray} is a one-way infinite path, in a graph. 
Its subrays are its \emph{tails}.
Two rays~$R_1$ and~$R_2$ of a given graph~$\Gamma$ are \emph{equivalent} if for every finite set~$S$ of vertices of~$\Gamma$ there is a component of~$\Gamma\setminus S$ which contains both a tail of~$R_1$ and of~$R_2$. 
The classes of the equivalent rays are called \emph{end}.
Diestel and K\"{u}hn~\cite{Ends} have shown that if we consider a locally finite graph~$\Gamma$ as a one-dimensional complex and endow it with the one complex topology then the topological ends of~$\Gamma$ coincide with the vertex-ends of~$\Gamma$.
For a graph~$\Gamma$ we denote the Freudenthal compactification of~$\Gamma$ by~$|\Gamma|$. 

A homeomorphic image of~$[0,1]$ in the topological space~$|\Gamma |$ is called \emph{arc}. 
A \emph{Hamilton arc} in~$\Gamma$ is an arc including all vertices of~$\Gamma$. 
By a \emph{Hamilton circle} in~$\Gamma$, we mean  a homeomorphic image of the unit circle in~$|\Gamma|$ containing all vertices of~$\Gamma$.
Note that Hamilton arcs and circles in a locally finite graph always contain all ends of the graph precisely once. 
A Hamilton arc whose image in a graph is connected, is a \emph{Hamilton double ray}.

\subsection{Graphs}
\label{on graphs}
All graphs in this paper are locally finite. 
As cycles are always finite, we need a generalization of Hamilton cycles for infinite graphs.  
We follow the topological  approach defined in Section~\ref{topology} by using the circles in the Freudenthal compactification~$| \Gamma |$ of a locally finite~$\Gamma$ graph  as ``infinite cycles".
We use the term \emph{circles} for ``infinite cycles" and by \emph{cycles} we mean finite cycles.
The compactification points~$\Omega \defi |\Gamma| \setminus \Gamma$ are the \emph{ends} of~$\Gamma$.
We say that an end~$\omega$ \emph{lives in a component~$C$} of~$\Gamma \setminus X$, where~$X$ is a finite subset of~$V(\Gamma)$ or a finite subset of~$E(\Gamma)$, when a ray of~$\omega$ has a tail completely contained in~$C$, and we denote~$C$ by~$C(X,\omega)$.\footnote{Note that the choice of the rays does not matter as~$X$ is finite and~$\Gamma$ is locally finite.} 
A sequence of finite edge sets~$(F_i) _{i \in \bbN}$ is a \emph{defining sequence} of an end~$\omega$ if~$C_{i+1} \subsetneq C_{i}$, with~${C_i \defi C(F_i,\omega)}$ such that~$\bigcap C_i = \emptyset$. 
We define the \emph{degree of an end~$\omega$} as the supremum over the number of edge-disjoint rays belonging to the class which corresponds to~$\omega$. 
We denote the set of ends of~$\Gamma$ by~$\Omega(\Gamma)$. 
A graph is called \emph{Hamiltonian} if it contains either a Hamilton cycle or a Hamilton circle.

Thomassen~\cite{ThomassenG2} defined a Hamilton cover of a finite graph~$\Gamma$ to be a collection of mutually disjoint paths~${P_1,\ldots,P_m}$ such that each vertex of~$\Gamma$ is contained in exactly one of the paths.
For easier distinction we call this a~\emph{finite Hamilton cover}.
An \emph{infinite Hamilton cover} of an infinite graph~$\Gamma$ is a collection of mutually disjoint \emph{double rays}, two way infinite paths, such that each vertex of~$\Gamma$ is contained in exactly one of them.
The order of an infinite Hamilton cover is the number of disjoint double rays in it.

In the context of Cayley graphs we switch  the notation of~\cite{commutatorsubgroup, wittesurvey} with the notation given in~\cite{diestelBook10noEE}.
The advantage of this is that uses labeled edges rather than just vertices.
We recall the definition in the following: 
Let~${G=\langle S\rangle}$ be a group and  let~$g$ and~$s_i\in S,~i \in \bbZ$, be elements of~$G$. 
In this notation~$g[s_1]^k$ denotes the concatenation of~$k$ copies of~$s_1$ from the right starting from~$g$ which translates to the path~$g,(gs_1),\ldots,(gs_1^k)$ in the usual notation. 
Analogously~${[s_1]^k g}$ denotes the concatenation of~$k$ copies of~$s_1$ starting again from~$g$ from the left.
In addition~$g[s_1,s_2,\ldots]$ translates to be the ray~$g,(gs_1),(gs_1s_2),\ldots$ and 
$$[ \ldots , s_{-2},s_{-1}]g[s_1,s_2, \ldots ]$$
translates to be the double ray 
$$\ldots, (gs_{-1}s_{-2}),(gs_{-1}),g,(gs_1),(gs_1s_2),\ldots$$ 
When discussing rays we extend the notation of~$g[s_1,\ldots, s_n]^k$ to~$k$ being countably infinite and write~$g[s_1,\ldots, s_n]^\bbN$ and the analogue for double rays.  
The statement that~$g[c_1,\ldots,c_k]$ is a cycle is short for saying that~$g[s_1,\ldots, s_{k-1}]$ is a path and that the edge~$s_k$ joins the vertices~${g s_1\cdots s_{k-1}}$ and~${g}$.

\subsection{Groups}
\label{on groups}
We note that we  suppose always that the generating set of~$G$ is symmetric, which means that if~$s\in S$, then~$s^{-1}\in S$.
We denote the Cayley graph of~$G$ with respect to~$S$ by~$\Gamma(G,S)$.  
For any two sets~$A$ and~$B$ we write~$A \sqcup B$ for the disjoint union of the sets~$A$ and~$B$. 
A finite group~$G$ is a \emph{$p$-group} if the order of each element of~$G$ is a power of~$p$, where~$p$ is a prime number.
Let~$A$ and~$B$ be two subsets of~$G$.
Then~$AB$ denotes the set~${\{ab\mid a\in A, b\in B \}}$ and hence~$A^2$ is defined as~$AA$. 
Let~$H \leq G$.
Then for~${g\in G}$ and~$h\in H$ we denote~$g^{-1}H g$ and~$g^{-1}h g$ by~$H^g$ and~$h^g$, respectively. 
An important subgroup of~$H$ is~$\mathsf{Core}(H):=\cap_{g\in G}H^g$ which is always normal in~$G$ and moreover if~${[G:H]=n}$, then the index~$\mathsf{Core}(H)$ in~$G$ is at most~$n!$, see~\cite[Theorem 3.3.5]{scott}.
We denote the order of the element~$g$ by~$o(g)$.
We denote the \emph{centralizer} of the element~$g$ by~${C_G(g)\defi \{h\in G\mid hg=gh\}}$ and the \emph{commutator subgroup} of~$G$  by~$G'$ which is the subgroup generated by all elements of the forms~$ghg^{-1}h^{-1}$.
Assume that~$H$ and~$K$ are two groups.
A \emph{short exact sequence} of the groups~$H,G$ and~$K$ is given by two following maps:
$$1\rightarrow H\xrightarrow{f} G\xrightarrow{g} K\rightarrow 1$$
Such that $f$ is injective and $g$ is surjective and moreover the kernel of~$g$ is equal to the image of $f$.
The group~$G$ is called an \emph{extension} of~$H$ by~$K$ if there exists the above short exact sequence.
We note that every semi-direct product~$H\rtimes K$ has a short exact sequence like above.
For a group~$G= \langle S \rangle$ we define~$e(G)\defi |\Omega(\Gamma(G,S)|$.
We note that this definition is independent of the choice of~$S$ as 
$$|\Omega(\Gamma(G,S))| = |\Omega(\Gamma(G,S^\prime))|$$ 
as long as~$S$ and~$S^\prime$ are finite, see~\cite[Theorem 11.23]{mei}.
Let~$H$ be a normal subgroup of~$G$.
Then we denote the set~$\{sH\mid s\in S\}$ by~$\overline{S}$.
We notice that~$\overline{S}$ generates~$\overline{G} \defi {G/H}$.
Let~$G_i$ be two groups with subgroups~$H_i$ where~${H_1\cong H_2\cong H}$, with~$i=1,2$.
A generating set~$S$ of~$G_1\ast_H G_2$ is called \emph{canonical} if~$S$ is a union of~$S_i$ for~$i=1,\ldots,3$ such that~$\langle S_i\rangle=G_i$ for~$i=1,2$ and~$H=\langle S_3\rangle$.
We note that when~$H=1$, then we assume that~$S_3=\emptyset$.
A subgroup~$H$ of~$G$ is called \emph{characteristic} if any automorphism~$\phi$ of~$G$ maps~$H$ to itself and we denote it by~$H\mathsf{char} G$.
A \emph{dihedral group} is defined with the presentation~$\langle a,b\mid b^2,a^n,(ba)^2\rangle$, where~$n\in \N\cup \infty$ and denote it by~$D_{2n}$.
By~$F_r$ we mean the free group of the rank~$r$, denoted by~$\mathsf{rank}(F_r)=r$.
Let~$A$ and~$B$ be two groups with subgroups~$C_1$ and~$C_2$ respectively such that there is an isomorphism~${\phi\colon C_1\to C_2}$.
The \emph{free product with amalgamation} is defined as

$$A_1\free_{C_1=C_2} B \defi  \langle S_1\cup S_2\mid R_1\cup R_2\cup C_1\phi^{-1}(C_1)\rangle.$$
For the sake of simplicity we denote it by $A\ast _C B$, where $C_1\cong C_2\cong C$. 
Next we define HNN-extension just for two-ended groups.
Let~$C=\langle S\mid R\rangle$ be a finite group and $\phi\in \mathsf{Aut}(C)$.
We now insert a new symbol~$t$ not in~$C$ and we define the \emph{HNN-extension} of~${C\ast_{C}}$ as follows:
$${C\ast_{C}} \defi \langle S, t\mid R\cup \{ t^{-1}c t \phi(c)^{-1} \mid \mbox{ for all } c \in C \}\rangle.$$
Finally we cite a famous structure theorem of groups in a formulation for two-ended groups which will turn out to be quite useful in proving our results. 
Let~$H$ be a subgroup of~$G$.
Then a subset~$T$ of~$G$  is called a \emph{left transversal} of~$H$ in~$G$  if~$T$  intersects every left coset of~$H$ at exactly one element.
Analogously we can define a \emph{right transversal} of~$H$ in~$G$.
\begin{thm}{\rm\cite[Theorem 4.1]{two-ended}}\label{stallings} 
	Let~$G$ be a finitely generated group. 
	Then the following statements are equivalent.
	\begin{enumerate}[\rm (i)]
		\item~$G$ is two-ended.
		\item~$G$ has an infinite cyclic subgroup of finite index.
		\item~$G=A {{\free\,\,}_{C}}B$ and~$C$ is finite and~$[A:C]=[B:C]=2$ or~$G=C\free_C$ with~$C$ is finite.
		\item $G/C$ is isomorphic to either $\Z$ or $D_{\infty}$, where $C$ is  finite normal subgroup.
	\end{enumerate}
\end{thm}

 A group~$G$ is called a \emph{planar group} if there exist a generating set~$S$ of~$G$ such that~$\Gamma(G,S)$ is a planar graph.
 We will only consider groups with locally finite Cayley graphs in this paper so we assume that all generating sets~$S$ will be finite.


\section{Constructing groups with Hamilton circles}
\label{construction}

One of the strongest results about the Lov\'{a}sz conjecture is the following  theorem which has been proved by Witte.

\begin{thm}{\rm\cite[Theorem 6.1]{p-group}}
\label{finite p-group}
Every connected Cayley graph on any finite~$p$-group is Hamiltonian.
\end{thm}	

In this section we are trying to present a generalization for Theorem~\ref{finite p-group} for infinite groups.
First of all we need to show that two-ended groups always contain a subgroup of index two.

\begin{lemma}\label{index 2}
Let~$G$ be a finitely generated two-ended group.
Then~$G$ contains a subgroup of index two.
\end{lemma}

\begin{proof}
It follows from  \cite[Lemma 11.31]{mei} and  \cite[Theorem 11.33]{mei} that there exists a subgroup~$H$ of index at most~$2$ together with a homomorphism~${\phi\colon H\to \Z}$ with finite kernel.
Now if~$G$ is equal to~$H$, then we deduce that~$G/K$ is isomorphic to~$\Z$ where~$K$ is the kernel of~$\phi$.
Let~$L/K$ be the subgroup of~$G/K$ corresponding to~$2\bbZ$.
This implies that the index of~$L$ in~$G$ is 2, as desired.
\end{proof}

 Now by Lemma \ref{index 2} we know that~$G$ always possesses a subgroup~$H$ of index~$2$.
In Theorem~\ref{index2 hamilton} we show that if any Cayley graph of~$H$ is Hamiltonian, then~$\Gamma(G,S)$ contains a Hamilton circle if~$S \cap H = \emptyset$. 

By the work of Diestel we get the following lemma as a tool to finding Hamilton circles in two-ended graphs. 

\begin{lemma}{\rm\cite[Theorem 2.5]{TopSurveyI}}\label{What is HC}
Let~$\Gamma=(V,E)$ be a two-ended graph. 
And let~$R_1$ and~$R_2$ be two doubles rays such that the following holds:
\begin{enumerate}[{\rm(i)}]
		\item~$R_1 \cap R_2 = \emptyset$
		\item~$V = R_1 \cup R_2$ 
		\item For each~$\omega \in \Omega(\Gamma)$ both~$R_i$ have a tail that belongs to~$\omega$. 
\end{enumerate}
Then~$R_1 \sqcup R_2$ is a Hamilton circle of~$\Gamma$.\qed
\end{lemma}

 For two-ended graphs we say~$R_1\sqcup R_2$ is a Hamilton circle if~$R_1$ and~$R_2$ fulfil the conditions of Lemma~\ref{What is HC}.  


The problem of finding Hamilton circles in graphs with more than two ends is a harder problem than finding Hamilton circles in graphs with one or two ends. 
For graphs with one or two ends the goal is to find one or two double rays containing all the vertices which behave nicely with the ends. 
For graphs with unaccountably many ends, it is not so straight forward to know what this desired structure could be. 
But the following powerful lemma by Bruhn and Stein helps us by telling us what such a structure looks like. 

\begin{lemma}\textnormal{\cite[Proposition 3]{degree}}
\label{What is HC Inf}
Let~$C$ be a subgraph of a locally finite graph~$\Gamma$. 
Then the closure of~$C$ is a circle if and only if the closure of~$C$ is topologically connected and every vertex or end of~$\Gamma$ in this closure has degree two in~$C$. \qed
\end{lemma}

\begin{thm}\label{index2 hamilton}
Let~$G=\langle S \rangle$ be a two-ended group with a subgroup~$H$ of index~$2$ such that~$H\cap S=\emptyset$. 
If any Cayley graph of~$H$ is Hamiltonian, then~$\Gamma(G,S)$ is also Hamiltonian. 	
\end{thm}

\begin{proof}
First we notice that~$H$ is two-ended, see \cite[Lemma 5.6]{ScottWall}.
Let~${g\in S}$. 
We claim that~$gS$ generates~$H$.
Since the index~$H$ in~$G$ is 2, we conclude that~$S^2$ generates~$H$.
So it is enough to show that~${\langle gS\rangle=\langle S^2\rangle}$.
In order to verify this we only need to show that~${s_is_j\in \langle gS\rangle}$, 
where~${s_i,s_j\in S}$.
Since both of~${gs_i^{-1}}$ and~$gs_j$ lie in~$gS$, we are able to conclude that~$s_is_j$ belongs to~$\langle gS\rangle$.
We now suppose that~${\mathcal R_1\sqcup \mathcal R_2}$ is a Hamilton circle in~${\Gamma(H,gS)}$.
Let
$$\mathcal R_i=[ \ldots , ss_{i_{-2}},ss_{i_{-1}}]g_i[ss_{i_1},ss_{i_2}, \ldots ],$$ 
where~$s_{i_j}\in S$ for~${i=1,2}$ and~${j\in\Z\setminus \{0\}}$.
Without loss of generality we can assume that~${g_1=1}$.
We will now ``expand" the double rays~$\Rcal_i$ to double rays in~$\Gamma(G,S)$. 
So we define 
$$\mathcal R'_i:=[ \ldots , s,s_{i_{-2}},s,s_{i_{-1}}]  g_i[s,s_{i_1},s,s_{i_2}, \ldots ]$$ 
for~${i=1,2}$.
We note that~$S\cap H=\emptyset$.
First we show that~$\Rcal_i^\prime$ really is a double ray. 
This follows directly from the definition of~$\Rcal_i^\prime$ and the fact that~$\Rcal_i$ is a double ray. 
It remains to show that~${\mathcal R'_1}$ and~${\mathcal R'_2}$ are disjoint and moreover their union covers each vertex of~$\Gamma(G,S)$.
Suppose that~$\Rcal_1^\prime$ and~$\Rcal_2^\prime$ meet, and let~$v \in~\Rcal_1^\prime \cap \Rcal_2^\prime$ with the minimal distance in~$\Rcal_1^\prime$ from the vertex~1. 
Now we have the case that~$v \in H$ or~$v \notin H$. 
Both cases directly give a contradiction. 
From~${v \in H}$ we can conclude that~$\Rcal_1$ and~$\Rcal_2$ meet, which contradicts our assumptions. 
Assume that~${v \notin H}$.
Without loss of generality we may assume that $1\neq v$.
Suppose that the path from~$1$ to~$v$ in~$\Rcal^\prime_1$.
This implies that~${v s^{-1} \in H}$ and~$v s^{-1}  \in \Rcal_1^\prime ,\Rcal_2^\prime$. 
But this contradicts both the minimality of the distance of~$v$ from~$1$ and the fact that~$v s^{-1} \in \Rcal_1,\Rcal_2$.
 
It remains to show that~$\Rcal_1^\prime$ and~$\Rcal_2^\prime$ each have a tail in each of the two ends of~$\Gamma(G,S)$. 
Let~$\omega$ and~$\omega^\prime$ be the two ends of~$\Gamma(G,S)$ and let~$X$ be a finite vertex set such that~$C(X,\omega) \cap C(X,\omega^\prime) = \emptyset$.  
It remains to show that~$\Rcal_i^\prime$ has a tail in both~$C(X, \omega)$ and~$C(X, \omega^\prime)$. 
By symmetry it is enough to show that~$\Rcal_i^\prime$ has a tail in~$C\defi C(X,\omega)$. 
Let~$C_H$ be the set of vertices in~$C$ which are contained in~$H$.
By construction of~$\Rcal_i^\prime$ we know that~$\Rcal_i^\prime \cap C_H$ is infinite, and as~$\Gamma(G,S)$ is infinite and as~$\Rcal_i^\prime$ is connected, we can conclude that~$C$ contains a tail of~$\Rcal_i^\prime$. 
\end{proof}

\begin{coro}
Let~$G$ be a two-ended group such that any Cayley graph of~$G$ is Hamiltonian.
If~$H=\langle S\rangle$ is any extension of~$G$ by~$\Z_2$ in such a way that~${S\cap G=\emptyset}$, then~$\Gamma(H,S)$ is Hamiltonian.\qed
\end{coro}

 With an analogous method of the proof of Theorem \ref{index2 hamilton}, one can prove the following theorem.

\begin{thm}\label{index2hamiltonarc}
Let~$G=\langle S \rangle$ be a two-ended group with a subgroup~$H$ of index 2 such that~$H\cap S=\emptyset$. 
If any Cayley graph of~$H$ contains a Hamilton double ray, then so does~$\Gamma(G,S)$.\qed
\end{thm}	

\begin{lemma}\label{arcZ}
Any Cayley graph of~$\bbZ$ contains a Hamilton double ray. 
\end{lemma}

\begin{proof}
Let~$\bbZ = \langle S \rangle$. 
We proof Lemma~\ref{arcZ} by induction on~$|S|$. 
For~$|S| = 2$ there is nothing to show. 
So we may assume that~$|S| \geq 2$ and any Cayley graph of~$\bbZ$ with less than~$|S|$ generators contains a Hamilton double ray. 
Let~$s \in S$ and define~$H \defi \langle S\setminus s \rangle$.
Because~$H$ is a subgroup of~$\Z$ we know that~$H$ is cyclic.   
By the induction hypothesis we know that there is a Hamilton double ray of~$H$, say~$R_H = [\ldots x_{-2},x_{-1}]1[x_1,x_2,\ldots]$. 
Let~$k \defi [\Z:H]$, note that~$k \in \bbN$. 
So we have~$G = \Sqcup_{i=0}^{k-1} H s^i$. 
We define 
$$ R \defi \cdots [s^-1]^{k-1}[x_{-2}] [s]^{k-1} [x_{-1}]1[s]^{k-1} [ x_1] [s^{-1}]^{-(k-1)} [x_2] \cdots~$$
As~$\bbZ$ is abelian we can conclude that~$R$ covers all vertices of~$\Gamma(G,S)$. 
It remains to show that~$R$ has tails in both ends of~$\Gamma(G,S)$ which follows directly from the fact that~$R_H$ is a Hamilton arc of~$H$ and the fact that the index of~$H$ in~$G$ is finite.  
\end{proof}

 Witte \cite{wittedigraphs} has shown that any Cayley graph of a finite dihedral group contains a Hamilton path. 
\begin{lemma}{\rm\cite[Corollary 5.2]{wittedigraphs}}\label{hamilton path}
Any Cayley graph of the finite dihedral group contains a Hamilton path.\qed
\end{lemma}

Next we extend the above mentioned lemma from a finite dihedral group to the infinite dihedral group. 

\begin{lemma}\label{arcDinf}
Any Cayley graph of~$\Dinf$ contains a Hamilton double ray.
\end{lemma}   

\begin{proof}
	Let~$S$ be an arbitrary generating set of~${\Dinf=\langle a,b \mid b^2 = (ab)^2 =1 \rangle}$.
Let~$S_1$ be a maximal subset of~$S$ in a such way that~$S_1\subseteq \langle a\rangle$ and set~$S_2:=S\setminus S_1$.
We note that each element of~$S_2$ can be expressed as~$a^jb$ which has order~$2$ for  every~$j\in \Z$.
First we consider the case that~$S_1$ is not empty.
Assume that~$H=\langle a^i\rangle$ is the subgroup generated by~$S_1$.
We note that~$H\mathsf{char}\langle a \rangle\unlhd \Dinf$ and so we infer that~$H\unlhd \Dinf$.
It follows from Lemma \ref{arcZ} that we have the following double ray~$\Rcal$:
$$[\ldots,s_{-2},s_{-1}]1[s_1,s_2,\ldots],$$
spanning~$H$ with each~$s_i\in S_1$ for~$i\in \Z\setminus \{0\}$.
We notice that~${\Dinf /H=\langle \overline{S_2}\rangle}$ is isomorphic to~$D_{2i}$ for some~$i\in\N$ and by Lemma~\ref{hamilton path} we are able to find a Hamilton path of~${\Dinf /H}$, say~${[x_1H,\ldots,x_{2i-1}H]}$, each~$x_{\ell}\in{ S_2}$ for $\ell\in\{1,\ldots,2i-1\}$.
On the other hand, the equality~$bab=a^{-1}$ implies that~${ba^tb=a^{-t}}$ for every~$t\in \Z$ and we deduce that~${xa^tx=a^{-t}}$ for every~${t\in \Z}$ and~${x\in \Dinf \setminus \langle a \rangle}$.\footnote{This follows as every element of~$\Dinf\setminus \langle a \rangle$ can be presented by~$a^ib$ for~$i\in \Z$.}
In other words, we can conclude that~$xs_ix=s_i^{-1}$ for each~$s_i\in S_1$ and~$x\in \Dinf \setminus \langle a \rangle$.
We now define a double ray~$\Rcal^\prime$ in~$\Dinf$ and we show that it is a Hamiltonian double ray.
In order to construct~$\Rcal'$, we define a union of paths. Set 
$$P_j:=p_j[x_1,\ldots,x_{2i-1},s_{j+1}^{-1},x_{2i-1},\ldots,x_1,s_{j+2}],$$
where~$p_j:=s_1\cdots s_{j}$ whenever~$j>0$ and~$p_0=1$ and~$p_j=s_{-1}\cdots s_j$ whenever~${j<0}$.
It is straight forward to see that~$P_{2j}$ and~$P_{2(j+1)}$ meet in exactly one vertex.
We claim that the collection of all~$P_{2j}$'s are pairwise edge disjoint for~${j\in \Z}$.
We only show the following case and we leave the other cases to the readers.
Assume that~$p_{2j}x_1\cdots x_{\ell}$ meets with~$p_{2j'}x_1\cdots x_{2i-1}s_{2j'+1}^{-1}x_{2i-1}\cdots x_{\ell'}$, where~$j<j'$ and~$\ell\leq\ell'$.
It is enough to verify~$\ell=\ell'$.
It is not hard to see that~$p_{2j}x_1\cdots x_{\ell}
=p_{2j'}s_{2j'+1}^{-1}x_1\cdots x_{\ell'}$.
We can see that the left hand side of the equality belongs to the coset~$Hx_1\cdots x_{\ell}$ and the other lies in~$Hx_1\cdots x_{\ell'}$ and so we conclude that~$\ell=\ell'$.
We  are now ready to define our desired double ray.
We define
$$\Rcal ^\prime:=\bigcup_{j\in \Z}P_{2j}.$$
It is straight forward to check~$\Rcal^\prime$ contains every element of~$\Dinf$, 
thus we conclude that~$\Rcal'$ is a Hamilton double ray, as desired.

If~$S_1$ is empty, then~${S\cap \langle a\rangle=\emptyset}$ and Theorem \ref{index2hamiltonarc} completes the proof.
\end{proof}

With a slight change to the proof of Lemma~\ref{arcDinf} we can obtain a Hamilton circle for~$\Dinf$. 

\begin{thm}
The Cayley graph of~$\Dinf$ is Hamiltonian for any generating set~$S$ with~${|S|\geq 3}$.
\end{thm}   

\begin{proof}
As this proof is a modification of the proof of Theorem~\ref{arcDinf}, we continue to use the notations of that proof here.	
We again assume that~$S_1 \neq \emptyset$.
Otherwise~$S\cap\langle a \rangle$ is empty.
Since $\langle a\rangle\cong \Z$, we are able to apply Theorem \cite[Theorem 3.1.3]{Cayley_HC_M_T} and conclude that the any Cayley graph of~$\langle a\rangle$ with the generating set of size greater than 2 is Hamiltonian.
We note that $|S|>2$ and so for every $g\in S$ we are able to infer that the Cayley graph of~$\langle a\rangle$ with respect to $gS$ is Hamiltonian.
In this case using Theorem~\ref{index2 hamilton} finishes the proof.
Thus we suppose that~$S_1$ is not empty.
Since~$|S|\geq 3$, each of~$P_j$ has length at least one.
Now we define new paths
$$P_j':=p_j[x_1,\ldots,x_{2i-2},s_{j+1}^{-1},x_{2i-2},\ldots,x_1,s_{j+2}],$$
where~$p_j:=s_1\cdots s_{j}$ whenever~$j>0$ and~$p_0=1$ and~$p_j=s_{-1}\cdots s_j$ whenever~$j<0$ and put 
$$\Rcal _1:=\bigcup_{j\in \Z}P'_{2j} \textnormal{ and } {\Rcal_2:=[\ldots,s_{-2}s_{-1}]x_{2i-1}[s_1,s_2,\ldots]}.$$
Now~$\Rcal_1\sqcup \Rcal_2$ is a Hamilton circle.
\end{proof}
%
%
The following lemma is a useful tool for finding Hamilton circles as it reduces the task from a global condition, like an infinite circle,  to something more local like finite cycles and matchings.

\begin{lemma}\textnormal{\cite[Lemma 3.2.2]{Cayley_HC_M_T}}
	\label{Cylinder}
	Let~$\Gamma$ be a graph that admits a partition of its vertex set into finite sets~$X_i, ~i \in \bbZ$, fulfilling the following conditions:
	
	\begin{enumerate}[\rm (i)]
		\item~$\Gamma[X_i]$ contains a Hamilton cycle~$C_i$ or~$\Gamma[X_i]$ is isomorphic to~$K_2$. 
		\item For each~${i \in \bbZ}$ there is a perfect matching between~$X_i$ and~$X_{i+1}$. 
		\item There is a~$k \in \bbN$ such that for all~$i,j \in \bbZ$ with~${|i -j| \geq k}$ there is no edge in~$\Gamma$ between~$X_i$ and~$X_{j}$.
	\end{enumerate}
	Then~$\Gamma$ has a Hamilton circle.\qed
\end{lemma}

We now give two lemmas which show that we can find normal subgroups in certain free-products with amalgamations or HNN-extensions. 

\begin{lemma}
	\label{free prod normal}
	Let~$G = G_1 \ast_H G_2$ be a finitely generated 2-ended group.
	 Then~$H$ is normal in~$G$.  
\end{lemma}

\begin{proof}
As~$G$ is two-ended we know that~$[G_i : H] =2$ for~$i \in \{1,2\}$.
Let~$g \in G$ be any element and~$G_1=H\sqcup Hg_1$ and~$G_2=H\sqcup Hg_2$
We benefit from the normal form presentation of each element, see  \cite[Theorem 11.3]{bogo} and we write~${g = h (g_1g_2)^\ell x}$ or~${g = h (g_2g_1)^\ell y}$ for some~$\ell \in \bbN\cup\{0\}$ and~$x \in \{1, g_1\}$. 
	Let~$f \in H$ be an arbitrary element. 
	We have to show that~$g f g^{-1}\in H$.
	This is equivalent to~$gf = \hat{h} g$ for some~$\hat{h} \in H$. 
    Let us assume that~$g=h (g_1g_2)^\ell$ where $\ell$ is the minimal number with this property.
	In other words for $\ell'<\ell$, we have $h(g_1g_2)^{\ell'}f(g_1g_2)^{-\ell'}\in H$.
	Since~$H$ is a normal subgroup in~$G_1$ and~$G_2$, we conclude that~$g_2f=f'g_2$ and~$g_1f'=f''g_1$, where~$f',f''\in H$.
	We conclude the following: 
	\begin{align*}
	gf &= h (g_1g_2)^\ell  f  \\
	&=  h (g_1g_2)^{\ell-1} g_1 f^\prime g_2 \\
	&= h (g_1g_2)^{\ell-1} f^{\prime \prime} g_1g_2 \\
   &= h \bar{f} (g_1g_2)^{\ell-1}g_1g_2 \mbox{ for some } \bar{f} \in H \\
	&= \bar{h} (g_1g_2)^{\ell} \mbox{ for some } \bar{h} \in H
	\end{align*}
	The other case is analogous to this case.
	This concludes the proof.
\end{proof}

\begin{lemma}
	\label{HNN normal}
	Suppose that~$G$ is a two-ended group which splits over a subgroup~$H=\langle S\mid R\rangle$ as an HNN-extension. 
	i.e. $${G  =\langle S,t\mid R,tht^{-1}=\phi(h)\text{ for every }h\in H\rangle},$$ where~${\phi\in \mathsf{Aut}(H)}$.
	Then~$H$ is normal in~$G$. 
\end{lemma}

\begin{proof}
	Let~$g \in G$.
	We have to show that~$g h g^{-1}\in H$ for~$g \in G$ and~${h \in H}$. 
	By our presentation of~$G$ we know that~${g= h_1 t^{i_1} \cdots h_nt^{i_n}}$.
	From~$tht^{-1} = \phi(h)=h'$ we conclude the following: 
	\begin{align*}
     t^2 h t^{-2} &= t h' t^{-1} \\
	&=  \phi(h')  \\
	&= h'' \mbox{ where } h'' \in H \\
	\Rightarrow t^2 k &= k^{\ell^2} t^2 
	\end{align*}
	By induction we obtain~$t^m ht^{-m}\in H$ for~$m \in \bbN$ and we can extend this by replacing~$t$ with~$t^{-1}$ to all~$m \in \bbZ$.  
	This implies that we have a presentation for each~${g \in G}$ as~${g= h t^m}$ for some~${m \in \bbZ}$ and~$h\in H$.
	Let~${f \in H}$ be given. 
	We conclude that
	\begin{align*}
	g f &= ht^mf = hf't^m \mbox{ for some } f'\in H \\
	&= f''ht^m  \mbox{ for some } f^{\prime \prime} \in H
	\end{align*}
	This finishes the proof.
\end{proof}

\begin{thm}
\label{Z_p2}
Let~$G=\langle S\rangle$ be a two-ended group which splits over ~$\Z_p$ such that~$S\cap \Z_p\neq \emptyset$, where~$p$ is a prime number.
Then~$\Gamma(G,S)$ has a Hamilton circle.
\end{thm}   

\begin{proof}
First we notice that~$S$ and~$\Z_p$ meet in exactly one element and its inverse, say~${S\cap\Z_p=\{k,k^{-1}\}}$.
Otherwise~$S$ is not minimal.
By Theorem~\ref{stallings} we already know that~$G$ is isomorphic to~${G_1\free_{\Z_p} G_2}$ or an HNN-extension of the subgroup~$\Z_p$, where~$|G_1|=|G_2|=2p$.
First assume that~$G\cong G_1\free_{\Z_p} G_2$, where~$G_i$ is a finite group such that~$[G_i:\Z_p]=2$ for~$i=1,2$.
It follows from Lemma \ref{free prod normal} that~$\Z_p$ is a normal subgroup of~$G$, and we deduce that~$G/\Z_p\cong \Z_2\ast\Z_2$ which is isomorphic to~$\Dinf$.
We set~$S':=S\setminus \{k,k^{-1}\}$ and now the subgroup generated by~$S'$ has only trivial intersection with~$\Z_p$.
Otherwise~$\bbZ_p \ni x \in \langle S^\prime \rangle$ yields that~$k \in \langle S^\prime \rangle$, which cannot happen as~$S$ was minimal.
We denote this subgroup by~$H$.
Note that~$H\Z_p=G$ because~$\bbZ_p$ is normal.\footnote{To illustrate: Consider the generating sets.  Because~$\langle {k } \rangle$ is normal in~$G$ we can conclude that~$G= \langle S \rangle = \langle S^\prime  \rangle \langle k \rangle$.}
So we can conclude that~$H$ is isomorphic to~$\Dinf \cong \bbZ_2 \ast \bbZ_2$ as:
$$\bbZ_2 \ast \bbZ_2 \cong (G_1 \free_{\bbZ_p} G_2 )/ \bbZ_p= G / \bbZ_p = (H \bbZ_P ) / \bbZ_p \cong H / (H \cap \bbZ_p) =H$$
It follows from  Lemma~\ref{arcDinf} that there exists the following Hamilton double ray~$\Rcal$ in~$H$:
$$[ \ldots , s_{-2},s_{-1}]1[s_1,s_2, \ldots ],$$
with~$s_i\in S'$.
We notice that~$\Rcal$ gives a transversal of the subgroup~$\Z_p$.
It is important to note that~$\Z_p=\langle k\rangle$ is a normal subgroup of~$G$.
Set~$x_i \defi \Pi_{j=1}^i  s_j$ for~$i \geq 1$ and~$x_i \defi \Pi_{j=1}^is_{-j}$ for~$i \leq -1$. 
There is a perfect matching between two consecutive cosets~$\Z_px_i$ and~$\Z_px_{i+1}$.
We now are ready to apply Lemma~\ref{Cylinder} to obtain a Hamilton circle.

Now assume that~$G$ is an HNN-extension which splits over~$\Z_p$.
We recall that~$G$ can be represented by~${\langle k,t\mid k^p=1,t^{-1}kt=\phi(k)\rangle}$, with~${\phi\in \mathsf{Aut}(\Z_p)}$.
Since~$\Z_p$ is a normal subgroup(see Lemma \ref{HNN normal}), we conclude that~$G=\Z_p\langle t\rangle$.
Again set~${S':=S\setminus\{k,k^{-1}\}}$ and~${H \defi \langle S^\prime \rangle}$.
$$ \langle S^\prime \rangle =  H = H / (H \cap \bbZ_p) \cong \bbZ_p H / \bbZ_p = G / \bbZ_p = \bbZ_p \langle t \rangle / \bbZ_p \cong \langle t \rangle.$$  
Hence we deduce that~$S':=S\setminus\{k,k^{-1}\}$ generates~$\langle t\rangle$.
It follows from Lemma \ref{arcZ} that~$\Gamma(\langle t\rangle, S')$ contains a Hamilton double ray.
By the same argument as in the other case we can find the necessary cycles and the matchings between them to use Lemma~\ref{Cylinder} to find the desired Hamilton circle.
\end{proof}   
In the following theorem we are able to drop the condition of~$S \cap H \neq \emptyset$ if~$p=2$.

\begin{thm}
	\label{split over 2group}
	Let~$G$ be a two-ended group which splits over~$\Z_2$.
	Then any Cayley graph of~$G$ is Hamiltonian.
\end{thm}

\begin{proof}
Suppose that~$G = \langle S \rangle$.
If~$S$ meets~$\Z_2=\{1,k\}$, then by Theorem \ref{Z_p2} we are done.
So we can assume that~$S$ does not intersect~$\Z_2$.
It follows from Lemma \ref{free prod normal} and Lemma \ref{HNN normal}~$\Z_2$ is a normal subgroup and we deduce from Theorem \ref{stallings} that~$\overline{G}=G/\Z_2$ is isomorphic to~$\Z$ or~$D_{\infty}$.
In either case we can find a Hamilton double ray in~${\Gamma(\overline{G},\overline{S})}$ by either  Lemma~\ref{arcZ} or  Lemma~\ref{arcDinf}, say
$$\overline{\Rcal}= [\ldots, \overline{s_{-1}}]1[\overline{s_1},\ldots].$$
This double ray induces a double ray in~$\Gamma(G,S)$, say
$${\Rcal}= [\ldots, s_{-1}]1[s_1,\ldots].
$$
We notice that~$\Rcal$ meets every coset of~$\Z_2$ in~$G$ exactly once.
We now define the following double ray 
$$\Rcal':=[\ldots,s_{-1}]k[s_1,\ldots].$$
It is important to note that~$\Rcal$ and~$\Rcal'$ do not intersect each other.
Otherwise there would be a vertex adjacent to two different edges with the same label and this yields a contradiction.
Now it is not hard to see that~$\Rcal \sqcup \Rcal'$ forms a Hamilton circle.
\end{proof}	

\begin{remark}
	The assumption that~$G$ is two-ended is necessary and it cannot be extended to multi-ended groups.
	For instance, consider~$G=\Z_6\ast_{\Z_2}\Z_6$.
	It is proved in {\rm\cite{Cayley_HC_M_T}} that there is a generating set~$S$ of~$G$ with~$|S|=3$ such that~$\Gamma(G,S)$ is not Hamiltonian.
\end{remark}

\section{Generalization of Rapaport Strasser }
\label{Section_rap}

In this section we take a look at the following famous theorem about Hamilton cycles of Cayley graphs of finite groups which is known as Rapaport Strasser's Theorem and generalize the 2-connected case to infinite groups in Theorem~\ref{inf_rap_k2}. 

\begin{thm}{\rm\cite{rap}}\label{Rapaport-Strasser}
	Let~$G$ be a finite group, generated by three involutions~$a,b,c$ such that~$ab=ba$.
	Then the Cayley graph~$\Gamma(G, \{a,b,c\})$ is Hamiltonian.
\end{thm}   
In the sequel, we will try to extend Theorem~\ref{Rapaport-Strasser} to infinite groups.
But we need to be careful.
There are nontrivial examples of infinite groups such that their Cayley graphs do not possess any Hamilton circle, see Section 4.1 of \cite{Cayley_HC_M_T}.
Here we have an analogous situation.
For instance let us consider~$\Z_2\ast (\Z_2\times \Z_2)$ with a canonical generating set.
Suppose that~$a$ is the generator of the first~$\Z_2$.
Then every edge with the label~$a$ in this Cayley graph is a cut edge.
Hence we only consider Cayley graphs of connectivity at least two. 
On the other hand our graphs are cubic and so their connectivities are at most three.

We note that by Bass-Serre theory, we are able to classify groups with respect to the low connectivity as terms of fundamental groups of graphs.
It has been done by Droms, see Section 3 of \cite{Droms}.
But what we need here is a presentation of these groups.
Thus we utilize the classifications of Georgakopoulos~\cite{Ageloscubic} to find a Hamilton circle.
First we need the following crucial lemma which has been proved by Babai.

\begin{lemma}{\rm\cite[Lemma 2.4]{babai1997growth}}\label{one-ended}
Let~$\Gamma$ be any cubic Cayley graph of any one-ended group.
Then~$\Gamma$  is~$3$-connected.
\end{lemma}	

By the work of Georgakopoulos in~\cite{Ageloscubic} we have the following lemma about the generating sets of 2-connected cubic Cayley graphs. 

\begin{lemma}{\rm\cite[Theorem 1.1 and Theorem 1.2]{Ageloscubic}}\label{kapa=2}
	Let~$G=\langle S \rangle~$ be a group, where~$S=\{a,b,c\}$ is a set of involutions and~$ab=ba$.
	If~$\kappa(\Gamma(G,S))=2$, then~$G$ is isomorphic to one of the following groups:  
	\begin{enumerate}[\rm (i)]
		\item~$\langle a,b,c\mid a^2,b^2,c^2,(ab)^2,(abc)^m\rangle$,~$m \geq 1$.
		\item~$\langle a,b,c\mid a^2,b^2,c^2,(ab)^{2},(ac)^m\rangle$,~$m\geq 2$.
	\end{enumerate}
\end{lemma}

With the help of the lemmas above we are able to prove the extension of Theorem~\ref{Rapaport-Strasser} for 2-connected graphs.

\begin{thm}
\label{inf_rap_k2}
	Let~$G=\langle S \rangle~$ be a group, where~$S=\{a,b,c\}$ is a set of involutions such that~$ab=ba$.
	If~$\kappa(\Gamma(G,S))=2$, then~$\Gamma(G,S)$ is Hamiltonian.
\end{thm}

\begin{proof}
Using Lemma~\ref{kapa=2} we can split the proof in two cases:
\begin{enumerate}[\rm (i)]
\item $G\cong \langle a,b,c\mid a^2,b^2,c^2,(ab)^2,(abc)^m\rangle$,~$m \geq 1$.\\
\noindent If~$m=1$, then~$G$ is finite and we are done with the use of Theorem~\ref{Rapaport-Strasser}.
	So we can assume that~$m\geq 2$.
We set~$\Gamma \defi \Gamma(G,\{a,b,c\})$. 
Let us define~$H \defi \langle ac\rangle$, note that~$H$ is isomorphic to~$\Z$ and let~$R$ be the double ray spanning~$H$.§
As~$H$ is a subgroup of~$G$ we can now cover~$\Gamma$ by disjoint copies of~$R$, set~$\Rcal$ as the union of all those copies of~$R$. 
We want to apply Lemma~\ref{What is HC Inf}. 
Obviously~$\Rcal$ induces degree two on every vertex of~$\Gamma$. 
It follows from transitivity, that for any end~$\omega$ there is a defining sequence~$(F_i)_{i \in \bbN}$ such that~$|F_i| =2$ and such that the label of each edge in each~$F_i$ is~$c$.

To illustrate, consider the following:
The cycle~$C \defi 1[a,b,a,b]$ separates~$\Gamma$ into two non-empty connected graphs, say~$\Gamma_1$ and~$\Gamma_2$.
Let~$e_1$ and~$e_2$ be the two edges of~$\Gamma$ between~$C$ and~$\Gamma_1$.
Note that the label of both of those edges is~$c$, additionally note that~$F \defi \{e_1,e_2\}$ separates~$\Gamma_1$ from~$\Gamma [\Gamma_2 \cup C]$.
Let~$R^\prime$ be any ray in~$\Gamma$ belonging to an end~$\omega$. 
There is an infinite number of edges contained in~$R^\prime$ with the label~$c$ as the order of~$a,b,ab$ and~$ba$ is two, let~$D$ be the set of those edges.  
We can now pick images under some automorphisms of~$F$ which meet~$D$ to create the defining sequence~$(F_i)_{i \in \bbN}$.  

The collection of double rays~$\Rcal$ meets every such~$F_i$ in exactly two edges. 
It is straight forward to check that~$\Rcal$ meets every finite cut of~$\Gamma$. 
This implies that the closure of~$\Rcal$ is topologically connected and that each end of~$\Gamma$ has degree two in this closure.

\item~$G\cong \langle a,b,c\mid a^2,b^2,c^2,(ab)^{2},(ac)^m\rangle$,~$m\geq 2$ \\
The proof of ii) is very similar to i). 
The element of infinite order here is~$bc$ and the defining sequence consists of two edges both with label~$b$ instead of~$c$.\footnote{One could also show that~$\Gamma$ is outer planer as it does not contain a~$K_4$ or~$K_{2,3}$ minor and thus contains a unique Hamilton circle, see the work of Heuer~\cite{Heuer_Hamilton3}.}\qedhere
\end{enumerate} 
\end{proof}

In the following we give an outlook on the problem of extending Theorem~\ref{Rapaport-Strasser} to infinite groups with 3-connected Cayley graphs. 
Similar to the Lemma~\ref{kapa=2} there is a characterization for 3-connected Cayley graphs which we state in Lemma~\ref{kapa=3}. 
Note that the items (i) and (ii) have at most one end. 

\begin{lemma}{\rm\cite{Ageloscubicplanar}}\label{kapa=3}
Let~$G=\langle S \rangle~$ be a planar group, where~$S=\{a,b,c\}$ is a set of involutions and~$ab=ba$.
If~$\kappa(\Gamma(G,S))=3$, then~$G$ is isomorphic to one of the following groups:  
	\begin{enumerate}[\rm (i)]
		\item~$\langle a,b,c \mid a^2, b^2, c^2, (ab)^2, (acbc)^m\rangle, m \geq 1$.
		\item~$\langle a, b, c \mid a^2, b^2, c^2, (ab)^2, (bc)^m, (ca)^p\rangle, m, p \geq  2$.
		\item~$\langle a, b, c \mid a^2, b^2, c^2, (ab)^2, (bcac)^n, (ca)^{2m}\rangle, n,m\geq 2$
	\end{enumerate}
\end{lemma}


Lemma~\ref{kapa=3} gives us hope that the following Conjecture~\ref{Conj_inf_rap} might be a good first step to prove Conjecture~\ref{Conj_agelos} of  Georgakopoulos and Mohar, see~\cite{Ageloscubicplanar}.

\begin{conj}
\label{Conj_inf_rap}
Let~$G$ be a group, generated by three involutions~$a,b,c$ such that~$ab=ba$ and such that~$\Gamma(G, \{a,b,c\})$ is two-connected. 
Then~$\Gamma(G, \{a,b,c\})$ is Hamiltonian.
\end{conj}

\begin{conj}{\rm\cite{Ageloscubicplanar}}\label{Conj_agelos}
Every finitely generated 3-connected planar Cayley graph admits a Hamilton circle.
\end{conj}


We hope that methods used to prove Conjecture~\ref{Conj_inf_rap}, and then possibly Conjecture~\ref{Conj_agelos}, would open the possibility to also prove additional results like the extension of Theorem~\ref{campan} of Rankin, which we  propose in Conjecture~\ref{Conj_campan}.

\begin{thm}{\rm\cite{campan}}
\label{campan}
Let~$G$ be a finite group, generated by two elements~$a,b$ such that~$(ab)^2 = 1$.
Then the Cayley graph ~$\Gamma(G, \{a,b\})$ has a Hamilton cycle.
\end{thm}    

\begin{conj}
\label{Conj_campan}
Let~$G = \langle S \rangle$ be a group, where~$S=\{a^\pm,b^\pm \}$ such that~${(ab)^2=1}$ and~${\kappa(\Gamma(G,S)) \geq 2}$.
 Then~$\Gamma(G,S)$ contains a Hamilton circle. 
\end{conj}


\section{Generating sets admitting Hamilton circles}
\label{Section_add_gen}
This section has two parts. 
In the first part we study the Hamiltonicity of Cayley graphs obtained by adding a generator to a given generating sets of a group.  
In the second part, we discuss an important theorem called the Factor Group Lemma which plays a key role in studying Hamiltonianicity of finite groups.

\subsection{Adding generators}
Fleischner proved in~\cite{Fleischner} that the square of every 2-connected finite graph has a Hamilton cycle. 
Georgakopoulos~\cite{AgelosFleisch} has extended this result to Hamilton circles in locally finite 2-connected graphs.
This result implies the following corollary: 

\begin{coro}\textnormal{\cite{AgelosFleisch}}
\label{agelos}
Let~$G=\langle S \rangle$ be an infinite group such that~$\Gamma(G,S)$ is two connected then~${\Gamma(G,S \cup S^2)}$ contains a Hamilton circle.\qed
\end{coro}
%
In the following we extend the idea of adding generators to obtain a Hamilton circle in the following manner. 
We show in Theorem~\ref{add_a} that under certain conditions, it suffices to add just a single new generator instead of adding an entire set of generators to obtain a Hamilton circle in the Cayley graph. 

\begin{theorem}
	\label{add_a}
	Let~$G=\langle S\rangle$ be a group with a normal subgroup~${H = \langle a \rangle}$ such that~$\Gamma(\overline{G},\overline{S}\setminus \{H\})$ has a Hamilton cycle.
	Then~${\Gamma(G, S \cup \{a^\pm\})}$ is Hamiltonian.  
\end{theorem}

\begin{proof}
We first notice that because~$\overline{G}$ contains a Hamilton cycle,~$G$ contains a cyclic subgroup of finite index and Theorem~\ref{stallings} implies that~$G$ is two-ended.
We  set~${\Gamma \defi \Gamma(G, S \cup \{a^{\pm}\})}$.
Assume that~${C =H [x_1,\ldots, x_n]}$ be the Hamilton cycle of~$\Gamma(\overline{G},\overline{S}\setminus \{H\})$. 
	As~$G$ is two-ended, we only need to find two disjoint double rays which together span~$\Gamma$ such that for every finite set~$X\subset V(\Gamma)$ each of those rays has a tail in each infinite component of~$\Gamma \setminus X$.
	By the structure of~$G$ we can write
	$$G =  \langle a \rangle \sqcup \bigsqcup_{i=1}^{n-1} \left ( \left ( \Pi_{j=1}^i x_j  \right ) \langle a \rangle \right ).$$
Let~$\Gamma^\prime$ be the subgraph of~$\Gamma$ induced by~$\bigsqcup_{i=1}^{n-1}   \left ( \Pi_{j=1}^i x_j   \right ) \langle a \rangle$.
We now show that there is a double ray~$R$ spanning~$\Gamma^\prime$ that has a tail belonging to each end.
Together with the double ray generated by~$a$ this yields a Hamilton circle. 
To find~$R$ we will show that there is a ``grid like" structure in~$\Gamma^\prime$.
One might picture the edges given by~$a$ as horizontal edges and we show that the edges given by the~$x_i$ are indeed vertical edges yielding a ``grid like'' structure. 

We claim that each~$x_i$ either belongs to~$C_G(a)$, i.e.~${a x_i = x_i a}$, or that we have the equality~${ax_i=x_i a^{-1}}$.
By the normality of~$\langle a\rangle$, we have~${a^g\in \langle a \rangle}$ for all~${g \in G}$.
In particular, if the order of~$x_i$ is not two, then we can find~${\ell,k\in \Z\setminus \{0\}}$ such that~$a^{(x_i^{-1})} = a^k$ and~$a^{x_i} = a^\ell$.\footnote{$a^{(x_i^{-1})} = x_i a x_i^{-1}  = a^k \Rightarrow a = x_i^{-1} a^k x_i = (x_i^{-1} a x_i)^k$ and with~$x_i^{-1} a x_i=a^{x_i} = a^\ell$ this implies~${a = (a^\ell)^k = a^{\ell k} \Rightarrow 1 = a^{\ell k -1.}}$}
Hence we deduce that~${1  =  a^{\ell k -1}}$.
It implies that~$k= \ell = \pm 1$ for each~$i$.  
For the sake of simplicity, we assume that~${k = \ell =1}$ for all~$i$. 
The other cases are totally analogous, we only have to switch from using~$a$ to~$a^{-1}$ in the appropriate coset in the following argument. 

Now suppose that~$o(x_i)=2$.
Since~$\langle a \rangle$ is a normal subgroup, then there exists~$j\in\Z$ such that~$a=x_ia^jx_i$.
In other words, we have~$a=(x_iax_i)^j$ and we again plug~$a$ in the right side of the preceding equality.
Thus the equation~$a=a^{j^2}$ is obtained.
Therefore~$j=\pm 1$, as the order of~$a$ is infinite and so the claim is proved.

Now we are ready to define the two double rays, say~$R_1$ and~$R_2$, which yield the desired Hamilton circle. 
For~$R_1$ we take~$\langle a \rangle$. 
To define~$R_2$ we first define a ray~$R_2^+$ and~$R_2^-$ which each starting in~$x_1$. 
Let 
\begin{align*}
	R_2^+ &\defi x_1[x_2,\ldots,x_{n-1},a,x_{n-1}^{-1},\ldots, x_2^{-1},a]^\bbN \\
	R_2^- &\defi x_1[a^{-1},x_2,\ldots,x_{n-1},a,x_{n-1}^{-1},\ldots, x_2^{-1},a]^\bbN \\
\end{align*}
By our above arguments, all those edges exist and we define~${R_2 \defi R_2^+ \cup R_2^-}$. 
By construction it is clear that~$R_1 \cap R_2 = \emptyset$ and~$V(\Gamma) \subseteq R_1 \cup R_2$. 
It also follows directly from construction that for both ends of~$G$ there is  a tail of~$R_i$ that belongs to that end.\end{proof}
Under the assumption that the Weak Lov\'{a}sz Conjecture holds true for finite Cayley graphs, we can reformulate Theorem~\ref{add_a} in the following way: 
\begin{coro}
	\label{add_a_with_Lovazs}
	For any two-ended group~$G=\langle S  \rangle$ there exists an~${a \in G}$ such that~$\Gamma(G, S\cup \{a^\pm\})$ contains a Hamilton circle.\footnote{This corollary remains true even if we only assume that every finite group contains a Hamilton path instead of a Hamilton cycle.}  
\end{coro}

\begin{proof}
	It follows from Theorem \ref{stallings} that~$G$ has a subgroup of finite index which is isomorphic to~$\Z$.
	We denote this subgroup by~$H$.
	If~$H$ is not normal, then we substitute~$H$ with~$\mathsf{Core}(H)$ which has a finite index as well.
	Now we are ready to invoke Theorem \ref{add_a} and we are done.
\end{proof}	 

\begin{coro}
	Let~$G = \langle S \rangle$ be a group and let~$G^\prime \cong \bbZ$ have a finite index in~$G$.
	Then there exists an element~$a \in G$ such that~$\Gamma(G,S\cup \{a^\pm \})$ has a Hamilton circle.\qed  
\end{coro}

One might be interested in finding a small generating set for a group such that the Cayley graph with respect to this generating set is known to contain a Hamilton cycle or circle. 
For finite groups this was done by Pak and Radoi\^{c}i\`c. 
\begin{theorem}
\label{pak}
\textnormal{\cite[Theorem 1]{HP_Cayley_Pak}}
Every finite group~$G$ of size~$|G| \geq 3$ has a generating  set~$S$ of size~$|S| \leq \log_2 |G|$, such that~$\Gamma(G,S)$ contains a Hamilton cycle. 
\end{theorem}
A problem with extending Theorem~\ref{pak} to infinite groups is that having a generating set of size at most~$\log_2$ of the size of the group is no restriction if the group is infinite. 
We only consider context-free groups and prevent the above problem by considering the index of the free subgroups in those context-free groups to obtain a finite bound for the size of the generating sets, see Theorem~\ref{inf_pak} for the details.  
Before we extend Theorem~\ref{pak} to infinite graphs we need some more lemmas. 
In the following we give an extension of Lemma~\ref{Cylinder} from two-ended graphs to graphs with arbitrarily many ends. 

\begin{lemma}
	\label{Inf Cylinder}
Let~$\Gamma^\prime$ be an infinite graph and let~$\Ccal^\prime$ be a Hamilton circle of~$\Gamma^\prime$. 
Let~$\Gamma$ be a graph fulfilling the following conditions:
	\begin{enumerate}[\rm (i)]
	\item~${V(\Gamma) = \Sqcup_{i=1}^k V(\Gamma_i^\prime)}$ where~$\Gamma_i^\prime$ are~$k$ pairwise disjoint copies of~$\Gamma^\prime$, with say~${0 < i \leq k}$.
	
	\item~$\Sqcup_{i=1}^k E(\Gamma_i^\prime) \subseteq E(\Gamma)$.
	
\item Let~$\Phi$ be the  natural projection of~$V(\Gamma)$ to~$V(\Gamma^\prime)$ and set~$[v]$ to be the set of vertices in~$\Gamma$ such that~$\Phi$ maps them to~$v$. 
Then for each vertex~$v^\prime$ of~$\Gamma^\prime$ there is
\begin{enumerate}[\rm(a)]
\item an edge between the two vertices in~$[v]$ if~$k=2$, or
\item a cycle~$C_v$ in~$\Gamma$ consisting exactly of the vertices~$[v]$ if~$k \geq 3$.
\end{enumerate}

\item There is a~$j \in \bbN$ such that in~$\Gamma$ there is no edge between vertices~$v$ and~$w$ if~$d_{\Gamma'}(\Phi(v),\Phi(w)) \geq j$.
		
	\end{enumerate}
	Then~$\Gamma$ has a Hamilton circle.
\end{lemma}

\begin{proof}
The proof of Lemma~\ref{Inf Cylinder} consists of two parts. 
First we extend the collection of double rays that~$\Ccal^\prime$ induces on~$\Gamma^\prime$ to a collection of double rays spanning~$V(\Gamma)$ by using the cycles~$C_v$.
Note that if~$k=2$, we consider the edge between the two vertices in each~$[v]$ as~$C_v$ as the circles found by~(ii)~(b) only are used to collect all vertices in~$[v]$ in a path, which is trivial if there are only two vertices in~$[v]$. 
In the second part we show how we use this new collection of double rays to define a Hamilton circle of~$\Gamma$. 
Let~$v^\prime$ and~$w^\prime$ be two vertices in~$\Gamma^\prime$ and let~$v_i$ and~$w_i$ be the vertices corresponding to~$v^\prime$ and~$w^\prime$ in~$\Gamma_i$.
If~$v^\prime w^\prime$ is an edge of~$\Gamma^\prime$ then by assumption (ii) we know that~$v_iw_i$ is an edge of~$\Gamma$ for each~$i$.
This implies that there is a perfect matching between the cycles~$C_v$ and~$C_w$. 

The Hamilton circle~$\Ccal^\prime$ induces a subgraph of~$\Gamma^\prime$, say~$\Rcal^\prime$.
As~$\Gamma^\prime$ is infinite, we know that~$\Rcal^\prime$ consists of a collection of double rays. 
Let 
$${R^\prime=\ldots , r_{-1},r_0,r_1,\ldots}$$
be such a double ray. 
Let~$R_1^\prime,\ldots, R_k^\prime$ be the copies of~$R^\prime$ in~$\Gamma$ given by assumption (i).
Let~$r^j_i$ be the vertex of~$R_j$ corresponding to the vertex~$r_i$.
We now use~$R^\prime$ to construct a double ray~$R$ in~$\Gamma$ that contains all vertices of~$\Gamma$ which are contained in any~$R^\prime_j$. 
We first build two rays~$R^+$ and~$R^-$ which together will contain all vertices of the copies of~$R^\prime$.

For~$R^+$ we start in the vertex~$r_0^1$ and take the edge~$r_0^1r_1^1$. 
Now we follow the cycle~$C_{r_1}$ till the next vertex would be~$r_1^1$, say this vertex is~$r^\ell_1$ and now take the edge~$r^\ell_1r^\ell_2$.
We repeat this process of moving along the cycles~$C_v$ and then taking a matching edge for all positive~$i$.
We define~$R^-$ analogously for all the negative~$i$ by also starting in~$r_0^1$ but taking the cycle~$C_{r_0}$ before taking matching edges. 
Finally we set~$R$ to be the union of~$R^+$ and~$R^-$.
As~$R^+ \cap R^- = r_0^1$ we know that~$R$ is indeed a double ray. 
Let~$\Rcal$ be the set of double rays obtained by this method from the set of~$\Rcal^\prime$. 

In the following we show that the closure of~$\Rcal$ is a Hamilton circle in~$|\Gamma|$.  
By Lemma~\ref{What is HC Inf} it is enough to show the following three conditions. 
\begin{enumerate}
	\item $\Rcal$ induces degree two at every vertex of~$\Gamma$, 
	\item the closure of~$\Rcal$ is topologically connected and
	\item every end of~$\Gamma$ is contained in the closure of~$\Rcal$ and has degree two in~$\Rcal$. 
\end{enumerate}
1.\ follows directly by construction. We can conclude 2.\ directly from the following three facts: First: Finite paths are topologically connected, secondly: there is no finite vertex separator separating any two copies of~$\Gamma^\prime$ in~$\Gamma$ and finally:~$\Rcal^\prime$ was a Hamilton circle of~$\Gamma^\prime$, and thus~$\Rcal^\prime$ meets every finite cut of~$\Gamma^\prime$ and hence~$\Rcal$ meets every finite cut of~$\Gamma$. 
It is straightforward to check that by our assumptions there is a natural bijection between the ends of~$\Gamma$ and~$\Gamma^\prime$.\footnote{Assumption (iv) implies that no two ends of~$\Gamma^\prime$ get identified and the remaining parts are trivial or follow from the Jumping Arc Lemma, see~\cite{diestelBook10noEE, TopSurveyI}.}
This, together with the assumption that the closure of~$\Rcal^\prime$  is a Hamilton circle of~$\Gamma^\prime$, implies 3.\ and thus the proof is complete. 
\end{proof}

We invoke {\rm\cite[Proposition 11.41]{mei}} and deduce the following useful fact

\begin{lemma}
\label{sub group quasi iso}
Let~$G$ be a finitely generated group with a subgroup~$H$ of finite index, then the numbers of ends of~$H$ and~$G$ are equal.
\end{lemma}

Now we want to invoke Lemma \ref{Inf Cylinder} in order to study context-free groups.
First of all let us review some basic notations and definitions regarding context-free groups.
A group~$G$ is called \emph{context-free} if~$G$ contains a free subgroup with finite index.
Let us have a closer look at context-free groups.
In the following,~$F$ will always denote a free group and~$F_r$ will denote the free group of rank~$r$.
So let~$F$ be a free subgroup of finite index of~$G$.
If~$F=F_1$, then~$G$ is two-ended, see Theorem \ref{stallings}.
Otherwise~$G$ has infinitely many ends, as the number of ends of~$G$ is equal to the number of ends of~$F$ by Lemma~\ref{sub group quasi iso}.
To extend Theorem~\ref{pak} to infinite groups we first need to introduce  the following notation. 
Let~$G$ be a context-free group with a free subgroup~$F_r$ with finite index.

It is known that~$\mathsf{Core}(F_r)$ is a normal free subgroup of finite index, see~\cite[Corollary 8.4, Corollary 8.5]{bogo}.
Here we need two notations.
For that let~$G$ be a fixed group. 
By~$m_H$ we denote the index of a subgroup~$H$ of~$G$, i.e.~${[G : H]}$.
We set
$$n_G \defi \min \{m_H \mid H\text{ is a normal free subgroup of } G \textnormal{ and } [G:H]<\infty\}$$
and
$$r_G\defi \min \{\mathsf{rank}(H) \mid H\text{ is a normal free subgroup of } G \textnormal{ and } n_G=m_H\}.$$
It is worth remarking that~$n_G\leq n!(r-1)+1$, because  we already know that~$\mathsf{Core}(F_r)$ is a normal subgroup of~$G$ with finite index at most~$n!$.
On the other hand, it follows from the Nielsen-Schreier Theorem, see \cite[Corollary 8.4]{bogo}, that~$\mathsf{Core}(F_r)$ is a free group as well and by Schreier’s formula (see \cite[Corollary 8.5]{bogo}), we conclude that the rank of~$\mathsf{Core}(F_r)$ is at most~$n!(r-1)+1$.

We want to apply Corollary~\ref{agelos} to find a generating set for free groups such that the corresponding Cayley graph contains a Hamilton circle. 
By a theorem of Geogakopoulos~\cite{AgelosFleisch}, one could obtain such a generating set~$S$ of $F_r$ by starting with the standard generating set, say~$S^\prime$, and then defining~${S \defi S^\prime \cup S^{\prime 2} \cup S^{\prime 3}}$.
Such a generating set has the size~${8r^3+4r^2 +2r}$.  
In Lemma~\ref{free_group_two_connected} we find a small generating set such that~$F_r$ with this generating set is 2-connected and obtain in Corollary~\ref{Free group HC} a generating set of size~$6r(r+1)$ such that the Cayley graph of~$F_r$ with this generating set contains a Hamilton circle. 

\begin{lemma}
\label{free_group_two_connected}
There exists a generating set~$S$ of~$F_r$ of size less than~$6r$ such that~$\Gamma(F_r,S)$ is two-connected.  
\end{lemma}

\begin{proof}
Let~$\{s_1,\ldots,s_r\}^{\pm}$ be the standard generating set of~$F_r$. 
We set
$$T \defi \{s_1,\ldots, s_r, s_1^2,\ldots, s_r^2, s_1s_2, s_1s_3,\ldots s_1s_r \}.$$
Finally we define~$S \defi T^\pm$. 
It is straightforward to see that~$|T| = 3r-1$ and hence~$|S| = 6r-2$.
We now claim that~$\Gamma \defi \Gamma(F_r,S)$ is 2-connected. 
For that we consider~$\Gamma \setminus\{1\}$ where~$1$ is the vertex corresponding to the neutral element of~$F_r$. 
It is obvious that the vertices~$s_i$ and~$s_i^{-1}$ are contained in the same component of~$\Gamma \setminus \{ 1\}$ as they are connected by the edge~$s_i^2$.
Additionally the edges of the form~$s_1s_i$ imply that~$s_1$ and~$s_i$ are always in the same component. 
This finishes the proof. 
\end{proof}
Using Lemma~\ref{free_group_two_connected} and applying Corollary~\ref{agelos} we obtain the following corollary. 

\begin{coro}
	\label{Free group HC}
	For every free group~$F_r$ there exists a generating set~$S$ of~$F_r$ of size at most~${6r( 6 r+1)}$ such that~$\Gamma(F_r,S)$ contains a Hamilton circle. \qed 
\end{coro}

We are now able to find a direct extension of Theorem~\ref{pak} for context-free groups.

\begin{thm}
	\label{inf_pak}
	Let~$G$ be a context-free group with~$n_G \geq 2$.
	Then there exists a generating set~$S$ of~$G$ of size at most~$\log_2 (n_G)+1 +6r_G(6r_G +1)$ such that~$\Gamma(G,S)$ contains a Hamilton circle.
\end{thm}

\begin{proof}
	Suppose that~$G$ is a context-free group.
	Furthermore let~$F_r$  be a free subgroup of~$G$ with finite index~$n$, where~$r\geq 1$.
	We split our proof into two cases.
	
	First assume that~$r=1$. 
	This means that~$G$ contains a subgroup isomorphic to~$\Z$ with finite index and thus~$G$ is two-ended. 
	Let~$H =\langle g \rangle$ be the normal free subgroup of~$G$ such that~${m_{\langle g \rangle} = n_G}$. 
	Let~$\overline{G} {\defi G / H}$. 
	We may assume~${|\overline{G}| \geq 3}$:
	by the assumptions we know that~${|\overline{G}| \geq 2}$, so if~${|\overline{G}| = 2}$ then we choose an element~$f \notin H$ and obtain a Hamilton circle of~$\Gamma \defi \Gamma(G, S^\pm)$ with~$S \defi \{f,g\}$ as~$\Gamma$ is isomorphic to the double ladder.  
	Our assumptions imply that~$\overline{G}$ is a group of order~$n_G$.
	As~$n_G$ is finite, we can apply Theorem~\ref{pak} to~$\overline{G}$ to find a generating set~$\overline{S}$ of~$\overline{G}$ such that~$\Gamma(\overline{G},\overline{S})$ contains a Hamilton cycle. 
	For each~$\bar{s} \in \overline{S}$ we now pick a representative~$s$ of~$\bar{s}$.
	Let~$S^\prime$ be the set of all those representatives. 
	We set~$S \defi S^\prime \cup \{g^{\pm}\}$. 
	By construction we know that~$G = \langle S \rangle$. 
	It is straightforward to check that~$\Gamma(G,S)$ fulfills the conditions of Lemma~\ref{Cylinder} and thus we are done as~$|S| =  \log_2(n_G) +2$.

	Now suppose that~$r\geq 2$.
	Let~$H$ be a normal free subgroup of~$G$ such that~${\mathsf{rank}(H)=r_G}$.
	By Corollary~\ref{Free group HC} we know that there is a generating set~$S_H$ of size at most~${6r_G(6r_G+1)}$ such that~${\Gamma_H \defi \Gamma(H,S_H)}$ contains a Hamilton circle. 
	
	If~${n_G =2}$ then, like in the above case, we can just choose an~$f \in G \setminus H$ and a set of representatives for the elements in~$S_H$, say~$S^\prime$, and set~$S \defi S^\prime \cup \{f^\pm\}$ to obtain a generating set such that~$\Gamma(G,S)$ fulfills the condition of Lemma~\ref{Inf Cylinder}.

	So let us assume that~${n_G \geq 3}$.  
	We define~${\overline{G} \defi G /H}$.
	As~$\overline{G}$ is a finite group we can apply Theorem~\ref{pak} to obtain a generating set~$\overline{S}$ for~$\overline{G}$ of size at most~$\log_2(n_G)$ such that~$\Gamma(\overline{G},\overline{S})$ contains a Hamilton cycle. 
	Again choose representatives of~$\overline{S}$ to obtain~$S^\prime$. 
	Let~$S \defi S^\prime \cup S_H$. 
	Note that
	$$|S| \leq 6r_G(6r_G +1) + \log_2(n_G).$$ 
	By construction we know that~$G = \langle S \rangle$. 
	Again it is straightforward to check that~$\Gamma \defi \Gamma(G, S)$ fulfills the conditions of Lemma~\ref{Inf Cylinder} and thus we are done. 
\end{proof}

\begin{coro}
	Let~$G$ be a two-ended group.
	Then there exists a generating set~$S$ of~$G$ of~$\log_2 (n_G) +3$ such that~$\Gamma(G,S)$ contains a Hamilton circle.	\qed
\end{coro}	

\begin{remark}
	We note that it might not always be best possible to use Theorem~\ref{inf_pak} to obtain a small generating set for a given context-free group. 
	The advantage about Theorem~\ref{inf_pak} compared to just applying Corollary~\ref{agelos} is that one does not need to ``square'' the edges between copies of the underlying free group. 
	This is a trade-off though, as the following rough calculation shows. 
	Suppose that~$\Gamma \defi \Gamma(G, S)$ where~$G$ is a context-free group. 
	Additionally assume that~$\Gamma$ is 2-connected, which is the worst for Theorem~\ref{inf_pak} when comparing Theorem~\ref{inf_pak} with a direct application of Corollary~\ref{agelos}.
	Applying Corollary~\ref{agelos} to~$\Gamma$ we obtain that~$\Gamma(G, S \cup S^2)$ is Hamiltonian.
	For instance, let~$F_r$ be a normal free subgroup of~$G$ with $r_G=r$ and~$[G:F_r]=n_G$.
	We now define~$S_F$ as the standard generating set of~$F_r$ and~$S_H$ as the representative of the cosets of~$F_r$.
	Then set~$S:=S_F\cup S_H$.
	We have
	\begin{align*}
		|S_F^2|  &= 4r^2=4r_G^2 \\
		|S_H S_F| = |S_F S_H|  &= 2r_G=2n_Gr_G \\
		|S_H^2|  &=  n_G^2. 
	\end{align*}
	Applying Corollary~\ref{agelos} yields a generating set of size~${4r_G^2 + 4r_Gn_G + n_G^2}$ while a a direct application of Theorem~\ref{inf_pak} yields a generating set of  size at most~$${\log_2 (n_G) +1+6r_G(6r_G +1)}.$$
	Thus which result is better depends the rank of the underlying free group and~$n_G$. 
\end{remark}

\subsection{Factor Group Lemma:}
\label{factorSection}
In this section we study extensions of the finite factor group lemma to infinite groups.
For that we first cite the factor group lemma:

\begin{thm}{\rm\cite[Lemma 2.3]{commutatorsubgroup}}\label{finite factor}
	Let~$G=\langle S\rangle$ be finite and let~$N$ be a cyclic normal subgroup of~$G$. 
	If~$[\bar{x_1},\ldots,\bar{x_n}]$  is a Hamilton cycle of~$\Gamma(G/N,\overline{S}\setminus \{N\})$ and the product~$x_1\cdots x_n$ generates~$N$, then~$\Gamma(G,S)$   contains a Hamilton cycle.\qed
\end{thm}

To be able to extend Theorem~\ref{finite factor}, we have to introduce some notation.
Let~$G$ be a group with a generating set~$S$ such that~$G$ acts on a set~$X$. 
The vertex set of the~$\textit{Schreier graph}$ are the elements of~$X$. 
We join two vertices~$x_1$ and~$x_2$ if and only if there exists~$s\in \{S\}$ such that~$x_1 = s x_2$. 
We denote the Schreier graph by~$\Gamma (G,S,X)$.

Suppose that~$X$ is the set of right cosets of a  subgroup~$H$ of~$G$. 
It is an easy observation that~$G$ acts on~$X$. 
Now we are ready to generalize the factor group lemma without needing the cyclic normal subgroup.  It is worth remarking that if we consider the trivial action of~$G$ on~$G$, we have the Cayley graph of~$G$ with respect to the generating~$S$, i.e.~$\Gamma(G,S,G)=\Gamma(G,S)$.

\begin{thm}
\label{inf_group_factor}
	Let~$G=\langle S\rangle$ be a  group and let~$H$ be a subgroup of~$G$ and let~$X$ be the set of left cosets of~$H$.  
	If~$1< [G:H]<\infty$ and~$[x_1,\ldots,x_n]$ is a Hamilton cycle of~$\Gamma(G,S,X)$ and the product~$x_1\cdots x_n$ generates~$H$, then we have the following statements.
	\begin{enumerate}[\rm (i)]
		\item If~$G$ is finite, then~$\Gamma(G,S)$ contains a Hamilton cycle.
		\item If~$G$ is infinite, then~$\Gamma(G,S)$ contains a Hamilton double ray. 
	\end{enumerate}
\end{thm}

\begin{proof}
	\begin{enumerate}[{\rm (i)}]
		\item Let us define~$ a\defi x_1\cdots x_n$. 
		Assume that~$[G:H]=\ell$. 
		We claim that~$C \defi 1[x_1,\ldots, x_n]^{\ell}$ is the desired Hamilton cycle of~$G$. 
		It is obvious that~$C$ contains every vertex of~$H$ atleast once. 
		Suppose that there is a vertex~$v \neq 1$ in~$C$ which is contained at least  twice in~$C$. 
		Say 
		$$v= a^{i_1}[x_1,\ldots, x_{i_2}] = a^{j_1} [x_1,\ldots, x_{j_2}] \mbox{ with } {i_1 \leq j_1 < l} \mbox{ and } i_2,j_2 < n.$$ 
		This yields that 
		\[ x_1 \cdots  x_{i_2} = a^{k} x_1 \cdots x_{j_2}  \mbox{ with } k \defi j_1-i_1 \geq 0\]
		As~$1$ and~$a^k$ are contained in~$H$, we may assume that~$i_2 = j_2$. 
		Otherwise~$x_1 \cdots x_{i_2}$ would belong to a different right coset of~$H$ as~$a^k x_1 \cdots x_{j_2}$ which would yield a contraction. 
		Thus we can now write
		$$x_1 \cdots x_{i_2} = a^{k} x_1 \cdots x_{j_2}$$
		 and it implies that~$k=0$.
		We conclude that~$C$ is indeed a cycle. 
		It remains to show that every vertex of~$\Gamma(G,S)$ is contained in~$C$. 
		So let~$v \in  V(\Gamma(G,S))$ and let~$H x_1\cdots x_k$  be the coset that contains~$v$. 
		So we can write~$v = h x_1 \cdots x_k$ with~$h \in H$. 
		As~$a$ generates~$H$ we know that~$h = a^j$. 
		So we can conclude that~$v = a^j x_1 \cdots x_i \in C$.
		So~$C$ is indeed a Hamilton cycle of~$G$. 
		
		\item The proof of (ii) is analogous to the above proof. 
		First notice that since~$G$ has a cyclic subgroup of finite index, we can conclude that~$G$ is two-ended by Theorem~\ref{stallings}.
		We now repeat the above construction with one small change. 
		Again define~${a \defi x_1 \cdots x_n}$. 
		As the order of~$a$ in~$H$ is infinite, we define~$C$ to be a double ray. 
		So let~$$C \defi[x_1^{-1},\ldots , x_n^{-1}]^\bbN1[x_1,\ldots, x_n]^{\bbN}.$$
		It is totally analogously to the above case to show that no vertex of~$\Gamma(G,S)$ is contained more than once in~$C$, we omit the details here. 
		It remains to show that every vertex of~$\Gamma(G,S)$ is contained in~$C$. 
		This is also completely analogue to the above case. \qedhere
	\end{enumerate}
\end{proof}

 Let us have a closer look at the preceding theorem.
As we have seen in the above proof the product~$x_1\cdots x_n$ plays a key role.
In the following we want to investigate a special case.
Suppose that~$G=\langle S\rangle$ is an infinite group with a normal subgroup~${H=\langle a\rangle}$ of finite index and moreover assume that~$G/H$ contains a Hamilton cycle~$1[x_1,\ldots,x_n]$.
Depending on the element~$x=x_1\cdots x_n$,  the following statements hold:
\begin{enumerate}[\rm (i)]
	\item If~$x=a$, then~$\Gamma(G,S)$ has a Hamilton double ray.
	\item  If~$x=a^2$, then~$\Gamma(G,S)$ has a Hamilton circle.
	\item If~$x=a^k$ and~$k\geq 3$, then~$\Gamma(G,S)$ has an infinite Hamilton cover of order~$k$.
\end{enumerate}
This yields us to conjecture the following: 
\begin{conj}
\label{Conj Hamilton cover}
Every two-ended transitive graph has a finite Hamilton cover.
\end{conj}
In 1983 Durnberger \cite{Durnberger1983} proved the following theorem:
\begin{thm}{\rm{\cite[Theorem 1]{Durnberger1983}}}
	\label{fin Z_p}
	Let~$G$ be a finite group with~$G'\cong\Z_p$.
	Then any Cayley graph of~$G$ contains a Hamilton cycle.\qed
\end{thm}
 This yields a natural question: What does an infinite group~$G$ with~${G'\cong\Z_p}$ look like?
\begin{lemma}
Let~$G$ be a finitely generated group such that~$|G'|<\infty$.
Then~$G$ has at most two ends. 
\end{lemma}	

\begin{proof}
Since~$G/G'$ is a  finitely generated abelian group, by \cite[5.4.2]{scott} one can see that~$G/G'\cong \Z^n\oplus Z_0$ where~$Z_0$ is a finite abelian group and ${n\in \N\cup\{0\}}$.
As the number of ends of~$\Z^n\oplus Z_0$ is at most two we can conclude that the number of ends of~$G$ is at most two by~\cite[Lemma 5.7]{ScottWall}.
\end{proof}
We close the paper with the following conjecture.
In the following conjecture, we propose an extension of Theorem~\ref{fin Z_p}.
Please note that the methods of the proof of Theorem~\ref{Z_p2} can be used to show the special case of Conjecture~\ref{Z_p} if the generating set does not have empty intersection with the commutator subgroup.
\begin{conj}\label{Z_p}
Let~$G$ be an infinite group with~$G'\cong\Z_p$.
Then any Cayley graph of~$G$ contains a Hamilton circle. 
\end{conj}

%

\bibliographystyle{plain}
\bibliography{collective.bib}

   \end{document}